\documentclass[12pt]{amsart}
\usepackage[margin=0.9in]{geometry}

\newcommand{\R}{\mathbb{R}}
\newcommand{\N}{\mathbb{N}}

\newtheorem{theorem}{Theorem}[section]

\newtheorem{corollary}[theorem]{Corollary}
\newtheorem{lemma}[theorem]{Lemma}

\usepackage{amssymb}

\renewcommand{\le}{\leqslant}
\renewcommand{\geq}{\geqslant}
\renewcommand{\ge}{\geqslant}

\begin{document}

\title[Gradient estimates]{Gradient estimates
\\ for a class of anisotropic nonlocal operators}

\thanks{This work has been supported by the Australian Research Council Discovery Project
N.E.W. ``Nonlocal Equations at Work''.}

\author{Alberto Farina}

\address{Alberto Farina:
LAMFA -- CNRS UMR 6140 --
Universit\'e de Picardie Jules Verne --
Facult\'e des Sciences --
33, rue Saint-Leu --
80039 Amiens CEDEX 1, France --
Email:
{\tt alberto.farina@u-picardie.fr}  
}

\author{Enrico Valdinoci}

\address{Enrico Valdinoci:
University of Western Australia --
Department of Mathematics and Statistics --
35, Stirling Highway -- WA 6009 Crawley,
Australia -- and --
Universit\`a degli Studi di Milano -- Dipartimento di Matematica --
Via Cesare Saldini, 50 -- 20133 Milano, Italy
-- and -- Istituto di Matematica Applicata e Tecnologie Informatiche --
Via Ferrata, 1 -- 27100 Pavia, Italy
-- Email:
{\tt enrico@mat.uniroma3.it} 
}

\begin{abstract}
Using a classical technique introduced by Achi E. Brandt for elliptic equations,
we study a general class of nonlocal equations obtained as a superposition
of classical and
fractional operators in different variables. We obtain
that the increments of the derivative of the solution in the direction
of a variable experiencing classical diffusion are controlled linearly, with
a logarithmic correction. {F}rom this, we obtain H\"older estimates
for the solution.
\end{abstract}


\maketitle

\section{Introduction}

In this paper, we will consider a general family of
nonlocal operators
built from classical and fractional Laplacians in different directions.
Namely, the whole of the space~$\R^n$ is divided into orthogonal subspaces
along which a possibly different order operator acts. These sectional operators can be either classical
or fractional, but at least
one of them (say, one involving the last coordinate) is assumed
to be of classical type. Our aim is to obtain regularity estimates
for the solution in this last variable and then to deduce global regularity results.
\medskip

The mathematical framework in which we work is the following.
We denote by~$\{e_1,\dots,e_n\}$ the Euclidean base of~$\R^n$.
Given a point~$x\in\R^n$, we use the notation
$$ x=(x_1,\dots,x_n)=x_1 e_1+\dots + x_n e_n,$$
with~$x_i\in\R$. 

We divide the variables of~$\R^n$ into $m$ subgroups of variables,
that is we consider~$m\in\N$ and~$N_1,\dots,N_m\in\N$, with
$N_1+\dots+N_{m-1}=n-1$ and~$N_m=1$.
For~$i\in \{1,\dots,m\}$, we use the notation~$N'_i:= N_1+\dots+N_i$,
and we take into account
the set of coordinates
\begin{equation}\label{JK905L:gh}
\begin{split}
& X_1 := (x_1,\dots, x_{N_1}) \in \R^{N_1}\\
& X_2 := (x_{N_1+1},\dots, x_{N_2'}) \in \R^{N_2}\\
& \vdots\\
& X_i := (x_{N_{i-1}'+1},\dots, x_{N_i'}) 
\in \R^{N_i}\\
& \vdots\\
& X_{m-1} := (x_{N_{m-2}'+1},\dots, x_{N_{m-1}'}) 
\in \R^{N_{m-1}}\\
{\mbox{and }}\qquad& X_m:=x_n.
\end{split}\end{equation}
Given~$i\in\{1,\dots,m-1\}$ and~$s_i\in (0,1]$, in this paper we study
the (possibly fractional) $s_i$-Laplacian 
in the $i$th set of coordinates~$X_i$
(the fractional case corresponds to the choice $s_i\in(0,1)$,
while the classical case reduces to~$s_i=1$).

To denote these operators acting on subsets of variables,
given~$y=(y_1,\dots,y_{N_i})\in \R^{N_i}$
it is useful to introduce the notion of
increment induced by~$y$
with respect to the $i$th set of coordinates in~$\R^n$,
that is one defines
\begin{equation}\label{jU6yhwedcG90:lk}
y^{(i)} := y_1
e_{N_{i-1}'+1}+\dots+
y_{N_i}
e_{N_i'} \in \R^n.\end{equation}
With this notation, one can define
the $N_i$-dimensional 
(possibly fractional) $s_i$-Laplacian
in the $i$th set of coordinates~$X_i$ as
\begin{equation}\label{FL} (-\Delta_{X_i})^{s_i} u(x) :=
\left\{
\begin{matrix}
-\partial^2_{x_{N_{i-1}'+1}} u(x)-\,\dots\,-
\partial^2_{x_{N_i'}} u(x)& {\mbox{ if }}s_i=1, \\
&\\ c_{N_i, s_i}
\displaystyle\int_{\R^{N_i}} \frac{2u(x)-u(x+y^{(i)})-u(x-y^{(i)})
}{|y^{(i)}|^{N_i+2s_i}}\,dy^{(i)}
& {\mbox{ if }}s_i\in(0,1),
\end{matrix}
\right. \end{equation}
The quantity~$c_{N_i,s_i}$ in~\eqref{FL}
is just a positive normalization constant, that is
\begin{equation}\label{cns}
c_{N,s}:=
\frac{2^{2s-1}\,\Gamma(s+\frac{N}{2})}{\pi^{\frac{N}{2}}\,|\Gamma(-s)|},\end{equation}
where~$\Gamma$ is the Euler's Gamma Function.
See for instance~\cite{land, silv, guida, bucur} and references
therein for further motivations and an introduction to fractional operators.
\medskip

In this paper we consider a
pseudo-differential operator, which is the sum of
(possibly) fractional Laplacians in the different coordinate directions~$X_i$,
with~$i\in\{1,\dots,m-1\}$, plus
a local second derivative in the direction~$x_n$.
The operators involved may have
different orders and they may be multiplied
by possibly different coefficients.
Without loss of generality, we will assume that the last coefficient
(that is the one related to the local variable)
is normalized to be~$1$.
That is, given~$a_1,\dots,a_{m-1}\ge0$,
we define
\begin{equation}\label{GAOP} L\,:=\,\sum_{i=1}^{m-1} a_i (-\Delta_{X_i})^{s_i} -\partial^2_{x_n}.\end{equation}
Of course, the operator $L$ comprises as particular cases
the classical Laplacian, the fractional Laplacian, and the sum of fractional Laplacians
or fractional derivatives in different directions. Since some of the~$a_j$'s may
vanish, the case of degenerate operators is also taken into account.\medskip

We observe that in many concrete applications,
different types of classical/anomalous diffusions
may take place in different reference variables:
a natural example occurs for instance when classical diffusion
involving space variables is considered together with the
anomalous diffusion arising from the transmission of genetic information,
see e.g.~\cite{GEN1, GEN2}, and these kinds
of phenomena can be fruitfully discussed with the aid of operators
such as the one in~\eqref{GAOP}.\medskip

To state our main result, it is convenient to
introduce the following domain notation.
Given~$r>0$, we denote by~$B^N_r$ the open ball
of~$\R^N$ centered at the origin and with radius~$r$.
Also,
given~$d_1,\dots,d_m>0$, we set~$d:=(d_1,\dots,d_m)$
and
$$ Q_d := B^{N_1}_{d_1} \times\dots\times
B_{d_{m-1}}^{N_{m-1}}
\times (-d_m,d_m) =\prod_{i=1}^m B_{d_i}^{N_i},$$
where in the latter identity
we used the convention that~$N_m:=1$.

Then,
given~$\kappa>0$,
we denote by~$Q_{d,\kappa}$ the dilation of~$Q_d$ of factor~$\kappa$
in the last coordinate (leaving the others put),
that is
$$ Q_{d,\kappa}:=
B^{N_1}_{d_1} \times\dots\times
B_{d_{m-1}}^{N_{m-1}}
\times (-\kappa d_m,\kappa d_m).$$
The main result of this article is a quantitative bound
on the continuity of the derivative of the solution with respect to the last coordinate.
This quantitative estimate is ``almost'' of Lipschitz type, in the sense
that the increment of the last derivative of the solution
is bounded linearly, up to a logarithmic correction. This estimate will
also be the cornerstone to prove additional regularity results such as H\"older estimates
for the last derivatives and for the solutions in any direction.
Thus, the core of the matter is the following result:

\begin{theorem}\label{MAIN}
Let $f : Q_d \to \R$ and $u : \R^n \to \R$ be a solution of $Lu = f$ in $Q_d$.
Then, for any $y\in \left( -\frac{d_m}{4},\, \frac{d_m}{4}\right)$,
$$ \left| \partial_{x_n} u(0,y)
-\partial_{x_n} u(0,-y)
\right| \le
\frac{8\,\|\partial_{x_n} u\|_{L^\infty(\R^n)}}{d_m}\,|y|+
2\kappa \,|y|\,\log\frac{2d_m}{|y|},$$
where
\begin{equation}\label{CH}\begin{split}& 
\kappa:=\frac43 \,\left( \|Lu\|_{L^\infty(Q_d)}+
\sum_{i=1}^m
\frac{\|u\|_{L^\infty(\R^n)}}{1-\left( \frac34\right)^s}\,
\frac{\tilde c_{N_i,s_i} }{d_i^{2s_i}}\right)\\
{\mbox{and }}\;& 
\tilde{c}_{N,s}:=
\frac{2^{s}\,\Gamma\left(s+1\right)\,\Gamma\left(\frac{N}2+s
\right)}{\Gamma\left(\frac{N}2\right)}.\end{split}
\end{equation}
\end{theorem}

Higher regularity results for different types of nonlocal
anisotropic operators
have been obtained in~\cite{SRO, ROV}. In particular,
in these articles
only operators with the same fractional
homogeneity were taken into account. Very recently,
in~\cite{FALL},
a regularity approach
for anisotropic operators  and for sums of anisotropic
fractional Laplacians with different homogeneities
has been taken into account, with methods different from the ones exploited
in this paper.
\medskip

Operators as the one studied here have been considered in~\cite{FV},
where a Lipschitz regularity result and a Liouville type theorem
have been established (in this sense, Theorem~\ref{MAIN}
can be seen as a higher regularity theorem with respect to
formula (7) in~\cite{FV}).\medskip

The method of proof of Theorem~\ref{MAIN} relies on an elementary,
but very deep, technique introduced in~\cite{brandt69} for the classical case
of the Laplacian. Roughly speaking, this method
is based on dealing with a family of additional variables
and an operator in this extended space. These additional variables
are chosen to take into account the increments of the solution
and the extended operator to preserve the right notion of solutions.
Then, one constructs barriers for this new operator, which in turn
provide the desired estimate on the original solution.

Of course, in the nonlocal case one has to construct new barriers
for the extended operator, since this operator also possesses nonlocal
features, and the new barriers,
differently from the classical case, must control the original solution
on the whole of the complement of the domain, and not only along the boundary.
\medskip

{F}rom Theorem~\ref{MAIN}, and the fact that
$$ \lim_{y\to0} |y|^\alpha \log|y| =0\qquad{\mbox{ for any }}\;\alpha\in(0,1),$$
we deduce that, for any $\alpha\in(0,1)$, any Lipschitz solution $u$
is $C^{1,\alpha}$ in the interior with respect to the variable $x_n$,
as stated explicitly in the next result:

\begin{corollary}\label{here}
Let $u:\R^n\to\R$ be a solution of~$ L_\star u=f$ in~$ B_1$. Then,
for any $\alpha\in(0,1)$,
$$ \| \partial_{x_n} u\|_{C^\alpha(B_{1/2})} \le C\,\Big(
\|f\|_{L^\infty(B_1)}+\|u\|_{L^\infty(\R^n)}\Big),$$
for some~$ C > 0$, that depends on~$\alpha$,
$ a_1,\dots,a_{m-1}$
$ s$ and~$N_1,\dots,N_{m-1}$.
\end{corollary}

Using the results of this paper, it is also possible to deduce
regularity results in all the variables. As an example,
we consider the operator~$L$ in the case in which~$s_1=\dots=s_{m-1}=:s\in(0,1)$,
namely
$$ L_\star \,:=\,\sum_{i=1}^{m-1} a_i (-\Delta_{X_i})^{s} -\partial^2_{x_n},$$
and we give the following result:

\begin{corollary}\label{CC:x}
Let~$u:\R^n\to\R$ be a solution of~$L_\star u=f$ in~$B_1$ and~$\tilde u(x_1,\dots,x_{n-1}):=
u(x_1,\dots,x_{n-1},0)$. Let~$\alpha\in(0,1)$ with~$\alpha+2s\not\in\N$.
Then
$$
\|\tilde u\|_{C^{\alpha+2s}(B_{1/2}^{n-1})}\le
C\,\Big(\| f\|_{C^{\alpha}(B_1)}
+\|\partial_{x_n} f\|_{L^\infty(B_1)}+\|u\|_{L^\infty(\R^n)}
+\| u\|_{C^{\alpha}(B_1)}\Big),$$
for some $C>0$.
\end{corollary}

We notice that when~$\alpha+2s>1$, the estimate in
Corollary~\ref{CC:x} provides continuity of the first derivative of the solution.
The organization of the rest of this paper is the following.
In Section~\ref{SEZ1}, we introduce an auxiliary function which serves
as barrier for our solution.
Interestingly, following\footnote{We remark
that Brandt's original barriers are modeled on second degree polynomials, while the ones exploited here are algebraically more complicated. Nevertheless, we believe that
there is a heuristic idea that can links the classical barriers to the new ones. Philosophically, the quadratic part of a second degree polynomial takes care of a ``constant'' right hand side of a second order equation, which is somehow ``the worst term'' in the class of bounded right hand sides. On the other hand, the linear part of a second order polynomial can be used to provide additional symmetries with respect to a section that divides the cube into two equal parts. The linear term is also responsible for the final estimate, since the quadratic part is negligible near the origin. The barriers constructed in this papers are based on these heuristic considerations: they recover the original barriers by
Brandt as~$s\nearrow1$ and they somehow preserve the geometric structures that we have discussed.}
an idea in~\cite{brandt69},
it is convenient to construct this 
barrier in an extended space. The additional variable plays the role of a translation
for the original solution and the first
barrier is constructed by superposing appropriate one-side
translations of the original solution, while the second barrier is a power-like function
that solves the equation with constant right-hand-side with a logarithmic modification
of a harmonic function in the translation coordinates.

The
proof of Theorem~\ref{MAIN} is completed in Section~\ref{SEZ2}, using the previous barrier
and the maximum principle.
Corollaries~\ref{here} and~\ref{CC:x} are proved in Sections~\ref{SEZ3}
and~\ref{SEZ4}, respectively.

\section{Building barriers}\label{SEZ1}

We use the notation~$X':=(X_1,\dots,X_{m-1})$
and define
$$ Q':= \left\{ (x',y,z)\in\R^{n+1} \;{\mbox{ s.t. }}\;
x'\in B^{N_1}_{d_1} \times\dots\times
B_{d_{m-1}}^{N_{m-1}},\quad
y\in\left(0,\frac{d_m}{4}\right)
\;{\mbox{ and }}\; z\in\left(0,\frac{d_m}{4}\right)
\right\}.$$
Then, we consider the extended operator defined as
$$ {\mathcal{L}}:=
\sum_{i=1}^{m-1} a_i (-\Delta_{X_i})^{s_i} 
-\frac1{2}\partial^2_{y}
-\frac1{2}\partial^2_{z}.$$
Also,
we use the standard notation~$r_+:=\max\{r,0\}$
for any~$r\in\R$ and,
for any~$(X',y,z)\in \R^{n+1}$, we define
$$ \phi(X',y,z):= \frac14\Big[ u(X',y_+ +z_+)
-u(X',y_+-z_+)-u(X',-y_++z_+)+u(X',-y_+-z_+)
\Big]. $$
The main properties of this barrier are listed below:

\begin{lemma}
For any~$(X',y,z)\in Q'$, we have that
\begin{equation}\label{p phi 1}
{\mathcal{L}}\phi(x',y,z)=
\frac14\left[ Lu(X',y+z)
-Lu(X',y-z)-Lu(X',-y+z)+Lu(X',-y-z)
\right]
\end{equation}
and
\begin{equation}\label{p phi 2}
\|{\mathcal{L}}\phi\|_{L^\infty(Q')}\le \|Lu\|_{L^\infty(Q_d)}
.\end{equation}
Also, 
\begin{equation}\label{p phi 3}
{\mbox{if either $y\le0$ or $z\le0$, then $\phi(X',y,z)=0$.}}
\end{equation}
Furthermore,
\begin{equation}\label{p phi 4}
\left| \phi\left( X',y,z\right)\right|\le 
\|\partial_{x_n} u\|_{L^\infty(\R^n)}\,\min\{y_+,z_+\}
.\end{equation}
\end{lemma}

\begin{proof} A direct calculation gives~\eqref{p phi 1}, which in turn implies~\eqref{p phi 2}.
Formula~\eqref{p phi 3} also follows by inspection.

As for~\eqref{p phi 4}, since the roles of~$y$ and~$z$ are the same,
we just prove the estimate when~$y_+\le z_+$. To this aim,
we use~\eqref{p phi 3} and we see that
\begin{eqnarray*}&&
\left| \phi\left( x',y,z\right)\right| =
\left| \phi\left( x',y,z\right)-\phi\left( x',0,z\right)\right|
\\
&&\qquad\le
\frac14\Big[ 
\left|u\left(X',y_++z_+\right)-
u\left(X',z_+\right)
\right|+
\left|u\left(X',y_+-z_+\right)-u\left(X',-z_+\right)
\right| \\
&&\qquad\qquad+
\left|u\left(X',-y_++z_+\right)-u\left(X',z_+\right)
\right|+
\left|u\left(X',-y_+-z_+\right)
-u\left(X',-z_+\right)\right|
\Big]\\
&&\qquad\le \|\partial_{x_n} u\|_{L^\infty(\R^n)}\,y_+,
\end{eqnarray*}
which implies~\eqref{p phi 4}.\end{proof}

Now, for any~$(X',y,z)\in \R^{n+1}$, we define
\begin{eqnarray*}
\psi(X',y,z)&:= &\sum_{i=1}^{m-1}
\frac{\|u\|_{L^\infty(\R^n)}}{1-\left( \frac34\right)^s}\,\left(1
-
\left(1-\frac{|X'_i|^2}{d^2_i}
\right)_+^s\right)
\\&&\qquad+\frac{4\,\|\partial_{x_n} u\|_{L^\infty(\R^n)}\;\,y_+ \,z_+}{d_m} +
{\kappa \,y_+\,z_+}\,\left|
\log\frac{2d_m}{y_+ + z_+}\right|,\end{eqnarray*}
where $\kappa$ and $\tilde{c}_{N,s}$ are as in \eqref{CH}.

We notice that the choice of~$\tilde{c}_{N_i,s_i}$ is made in such a way that
\begin{equation}\label{choice}
(-\Delta_{X_i})^{s_i}\left(1-\frac{|X'_i|^2}{d^2}\right)^{s_i}_+
=\frac{\tilde{c}_{N_i,s_i}}{d^{2s_i}},
\end{equation}
see e.g. Table~3 in~\cite{dyda}.
On the other hand, the definition of~$\kappa$ makes it sufficiently large
to let~$\psi$ dominate~$\phi$, as stated in the following
result:

\begin{lemma}\label{PCO}
We have that
\begin{eqnarray}
\label{L120-1}
&& \psi\pm\phi\ge0 \;{\mbox{ in }}\;\R^{n+1}\setminus Q'\\
\label{L120-2}
{\mbox{and }}\;&&
{\mathcal{L}}(\psi\pm\phi)\le0 \;{\mbox{ in }}\; Q'.
\end{eqnarray}
\end{lemma}

\begin{proof} We remark that
the complement of~$Q'$
can be written as~$ P_1\cup P_2\cup P_3\cup P_4\cup P_5$,
where
\begin{eqnarray*}
&& P_1:= \left\{ |X_i'|\geq\frac{d_i}2
{\mbox{ for some }} i\in\{1,\dots,m-1\}\right\},\\
&& P_2:=\{y\le0\},\\
&& P_3:=\{z\le0\},\\
&& P_4:= \left\{y\ge\frac{d_m}{4}\right\},\\
{\mbox{and }}&& P_5:= \left\{z\ge\frac{d_m}4\right\}
.
\end{eqnarray*}
Now, on~$P_1$,
\begin{eqnarray*}
\psi\ge
\frac{\|u\|_{L^\infty(\R^n)}}{1-\left( \frac34\right)^s}\,
\left(1-
\left(1-\frac{1}{4}\right)^s\right) 
= \|u\|_{L^\infty(\R^n)} \ge |\phi|.\end{eqnarray*}
Also, on~$P_2\cup P_3$, using~\eqref{p phi 3} we see that~$\psi\ge0=|\phi|$.
In addition, recalling~\eqref{p phi 4}, in~$P_4$ we have that
$$ |\phi|
\le \|\partial_{x_n} u\|_{L^\infty(\R^n)}\,z_+\le
\frac{4\|\partial_{x_n} u\|_{L^\infty(\R^n)}}{d_m}\,y_+\,z_+\le \psi,$$
and a similar computation holds in~$P_5$.
By collecting these estimates, the claim in~\eqref{L120-1} plainly follows.

Now we observe that in $Q'$ we have that
$$ \frac{2d_m}{y_+ + z_+} = \frac{2d_m}{y+z} \ge 1$$
and consequently
$$ \left|
\log\frac{2d_m}{y_+ + z_+}\right| =
\log\frac{2d_m}{y+ z}.$$
Also,
$$ \frac{\partial^2}{\partial y^2}
\left( yz\,\log\frac{2d_m}{y+z} \right)=
-\frac{2z}{y+z}+\frac{yz}{(y+z)^2}$$
and therefore
\begin{equation}\label{q-1}
\left(\frac{\partial^2}{\partial y^2}
+\frac{\partial^2}{\partial z^2}
\right) \left( yz\,\log\frac{2d_m}{y+z} \right)
=-2+\frac{2yz}{(y+z)^2}.\end{equation}
Now we point out that, for any~$y$, $z>0$,
\begin{equation}\label{q-2}
\frac{yz}{(y+z)^2} \le\frac14.
\end{equation}
For this, for any~$t\geq0$, we set
$$ h(t):=
\frac{t}{(1+t)^2}$$
and we have that
$$ \sup_{t\ge0} h(t)=h(1)=\frac14.$$
Then, we find that
$$ \frac{yz}{(y+z)^2}
=h\left(\frac{z}{y}\right)\le\frac14,$$
which is~\eqref{q-2}.

{F}rom~\eqref{q-1} and~\eqref{q-2} we obtain that
$$ \left(\frac{\partial^2}{\partial y^2}
+\frac{\partial^2}{\partial z^2}
\right) \left( yz\,\log\frac{2d_m}{y+z} \right)\le -2
+\frac12=-\frac{3}{2}.$$
Therefore, in view of~\eqref{choice}, in~$Q'$ it holds that
$$ {\mathcal{L}}\psi \le
\sum_{i=1}^m
\frac{\|u\|_{L^\infty(\R^n)}}{1-\left( \frac34\right)^s}\,
\frac{\tilde c_{N_i,s_i} }{d_i^{2s_i}}
-\frac{3}{4}\,\kappa.$$
Hence, our choice of~$\kappa$ in \eqref{CH} implies that
$$ {\mathcal{L}}\psi \le\|Lu\|_{L^\infty(Q_d)}.$$
This, together with~\eqref{p phi 2}, proves~\eqref{L120-2}.\end{proof}

\section{Completion of the proof of Theorem~\ref{MAIN}}\label{SEZ2}

By Lemma~\ref{PCO} and the maximum principle
(see e.g. formula~(22) in~\cite{FV}), we have that~$\psi\pm\phi\ge0$
in~$Q'$, that is~$|\phi|\le\psi$ in~$Q'$.
We write this inequality at~$X'=0$, $y>0$,
divide by~$z>0$ and pass to the limit: we find that
\begin{eqnarray*}
&&\frac1{2}\left| \partial_{x_n} u(0,y)
-\partial_{x_n} u(0,-y)
\right| \\
&=&\lim_{z\searrow0}\frac1{4z}\left| u(0,y+z)
-u(0,y-z)-u(0,-y+z)+u(0,-y-z)
\right| \\&\le&
\frac{4\,\|\partial_{x_n} u\|_{L^\infty(\R^n)}}{d_m}\,y+
\kappa \,y\,\log\frac{2d_m}{y}
.\end{eqnarray*}
This establishes Theorem \ref{MAIN}.

\section{Proof of Corollary \ref{here}}\label{SEZ3}

To prove Corollary \ref{here}, we first give a preliminary
result that follows directly from Theorem~\ref{MAIN}:

\begin{corollary}\label{here2}
Let $u:\R^n\to\R$ be a solution of~$ Lu=f$ in~$ B_1$. Then,
for any $\alpha\in(0,1)$,
\begin{equation} \label{quasi11}
\| \partial_{x_n} u\|_{C^\alpha(B_{1/2})} \le C\,\Big(
\|f\|_{L^\infty(B_1)}+\|u\|_{L^\infty(\R^n)}+
\|\partial_{x_n} u\|_{L^\infty(\R^n)}\Big),\end{equation}
for some~$ C > 0$, that depends on~$\alpha$,
$ a_1,\dots,a_{m-1}$
$ s_1,\dots,s_{m-1}$ and~$N_1,\dots,N_{m-1}$.
\end{corollary}

With this result, 
an elementary, but useful, cut-off\footnote{The main
difference between~\eqref{quasi11} and~\eqref{GGFF}
is that the norm of~$\partial_{x_n} u$ gets localized
in the second formula.}
argument, gives that:

\begin{corollary}\label{Cut1}
Let $u:\R^n\to\R$ be a solution of~$ Lu=f$ in~$ B_1$. Then,
for any $\alpha\in(0,1)$,
\begin{equation}\label{GGFF}
\| \partial_{x_n} u\|_{C^\alpha(B_{1/4})} \le C\,\Big(
\|f\|_{L^\infty(B_{1/2})}+\|u\|_{L^\infty(\R^n)}+
\|\partial_{x_n} u\|_{L^\infty(B_{1/2})}\Big),\end{equation}
for some~$ C > 0$, that depends on~$\alpha$,
$ a_1,\dots,a_{m-1}$
$ s$ and~$N_1,\dots,N_{m-1}$.
\end{corollary}

\begin{proof}
Let $\tau_o\in C^\infty(\R)$
be such that $\tau_o(r)=1$ if $|r|\le 3/4$ and $\tau_o(r)=0$
if $|r|\ge 4/5$. For any $i\in\{1,\dots,m\}$ and any $X_i\in\R^{N_i}$,
we set $\tau_i(X_i):=\tau_o(|X_i|)$. Let also
\begin{eqnarray*}
&& \tau(x)=\tau(X_1,\dots,X_m):=\tau_1(X_1)\dots\tau_m(X_m)\\ {\mbox{and }}&&
v(x):=\tau(x)\,u(x).
\end{eqnarray*}
Notice that $v=u$ in $B_{3/4}$.
Also, if $x=(X_1,\dots,X_m)\in B_{1/2}$ and $|y^{(i)}|\le\frac1{10}$,
we have that $|X^{(i)}+y^{(i)}|\le \frac34$ and $|X^{(j)}|\le 3/4$,
that gives $\tau(x)=1$.
Thus, for any $x\in B_{1/2}$ and for any $i\in\{1,\dots,m-1\}$,
if $s_i\in(0,1)$,
\begin{equation}\label{5wertfyguhjsgcxfscds}
\begin{split}
&\int_{\R^{N_i}} \frac{v(x)-v(x+y^{(i)})}{|y^{(i)}|^{N_i+2s_i}}\,dy^{(i)}
= \int_{\R^{N_i}} \frac{u(x)-v(x+y^{(i)})}{|y^{(i)}|^{N_i+2s_i}}\,dy^{(i)}\\
&\qquad
=\int_{\R^{N_i}} \frac{u(x)-u(x+y^{(i)})}{|y^{(i)}|^{N_i+2s_i}}\,dy^{(i)}+g_i(x)\\
&\qquad{\mbox{ where }} g_i(x):=
\int_{|y^{(i)}|\ge \frac1{10} } 
\frac{(1-\tau)u(x+y^{(i)})}{|y^{(i)}|^{N_i+2s_i}}\,dy^{(i)}.\end{split}\end{equation}
Notice that, for any $x\in B_{1/2}$ and for any $i\in\{1,\dots,m-1\}$,
\begin{equation} \label{gmei}
|g_i(x)|\le
\int_{|y^{(i)}|\ge \frac1{10}}
\frac{\|u\|_{L^\infty(\R^n)}}{|y^{(i)}|^{N_i+2s_i}}\,dy^{(i)}
\le C\,\|u\|_{L^\infty(\R^n)},\end{equation}
for some $C>0$.
In addition, in $B_{1/2}$,
$$ \partial^2_{x_n} v = 
\partial^2_{x_n}u.$$
As a consequence of this and \eqref{5wertfyguhjsgcxfscds}, in $B_{1/2}$ we have that
\begin{equation}\label{4567ihgzf} L v = g,\end{equation}
with
\begin{equation} \label{gme}
g(x):= f(x)+
\sum_{{1\le i\le m-1}\atop{s_i\in(0,1)}}b_i g_i(x),\end{equation}
for suitable $b_1,\dots,b_{m-1}$.
We remark that
$$ \|g\|_{L^\infty(B_{1/2})} \le C\,\Big(\|f\|_{L^\infty(B_{1/2})} 
+\|u\|_{L^\infty(\R^n)}\Big),$$
up to renaming $C>0$, thanks to \eqref{gmei}.
{F}rom this, \eqref{gme} and Corollary \ref{here2} we obtain that
\begin{eqnarray*}
\| \partial_{x_n} u\|_{C^\alpha(B_{1/4})} &=&
\| \partial_{x_n} v\|_{C^\alpha(B_{1/4})}
\\ &\le& C\,\Big(
\|g\|_{L^\infty(B_{1/2})}+\|v\|_{L^\infty(\R^n)}+
\|\partial_{x_n} v\|_{L^\infty(\R^n)}\Big)\\ &\le& C\,\Big(
\|f\|_{L^\infty(B_{1/2})}+\|u\|_{L^\infty(\R^n)}+
\|\partial_{x_n} u\|_{L^\infty([-1,1]^n)}\Big),\end{eqnarray*}
up to renaming $C>0$. This is the desired result, up to resizing balls.
\end{proof}

Now we recall
that an estimate for~$\| \partial_{x_n}u\|_{L^\infty(B_{1/2})}$
has been given in formula~(8) of~\cite{FV}. {F}rom this and \eqref{GGFF},
the claim in Corollary \ref{here} plainly follows, up to resizing balls.

\section{Proof of Corollary~\ref{CC:x}}\label{SEZ4}

The argument combines some techniques from Corollary 1.3 in~\cite{FV}
and Theorem~1.1(b) in~\cite{SRO}, together\footnote{We take this
opportunity to correct a flaw in the statement of Corollary 1.3 in~\cite{FV}.
Namely, the condition 
$$ \gamma:=\left\{
\begin{matrix}
2s & {\mbox{ if }} s\ne 1/2,\\
1-\epsilon & {\mbox{ if }} s=1/2
\end{matrix}
\right.$$
has to be replaced by
$$ \gamma:=\left\{
\begin{matrix}
2s & {\mbox{ if }} s< 1/2,\\
1-\epsilon & {\mbox{ if }} s\ge1/2
\end{matrix}
\right.$$
and the two lines after the statement can be deleted. The correct
statement of Corollary 1.3 in~\cite{FV}
is the one in the arxiv version \cite{FVGIU}}
with Corollary~\ref{here}
here.

To prove Corollary \ref{CC:x} we start with a preliminary and global version
of it:

\begin{lemma}\label{39ygccvghfi11vbdhewi}
Let~$\alpha\in(0,1)$ with~$\alpha+2s\not\in\N$.
Let~$u:\R^n\to\R$ be a solution of~$L_\star u=f$ in~$B_1$.

Let $\tilde u(x_1,\dots,x_{n-1}):=
u(x_1,\dots,x_{n-1},0)$.

Then
\begin{eqnarray*}
\|\tilde u\|_{C^{\alpha+2s}(B_{3/4}^{n-1})}&\le& C\,\Big(
\| f\|_{C^{\alpha}(B_{4/5})}
+\|\partial_{x_n} f\|_{L^\infty(B_1)}
\\ &&\quad+
\| u\|_{C^{\alpha}(\R^{n-1}\times [-{1}/{100},{1}/{100}])}
+\|\partial_{x_n} u\|_{L^\infty(\R^n)}
\Big)
,\end{eqnarray*}
for some $C>0$.
\end{lemma}

\begin{proof}
Given~$\tau\in\R$, with~$|\tau|$ sufficiently small,
we set
$$ u^{(\tau)}(x):= \frac{u(x+\tau e_n)-u(x)}{\tau}
\quad{\mbox{ and }}\quad f^{(\tau)}(x):= \frac{f(x+\tau e_n)-f(x)}{\tau}.$$
We remark that one can bound~$\|u^{(\tau)}\|_{L^\infty(\R^n)}$
with~$\|\partial_{x_n} u\|_{L^\infty(\R^n)}$.
Similarly, one can bound~$
\|\partial_{x_n} u^{(\tau)}\|_{L^\infty(\R^n)}$
with~$\|\partial_{x_n}^2 u\|_{L^\infty(\R^n)}$.

Notice also that~$L_\star u^{(\tau)}=f^{(\tau)}$ in $B_1$,
and thus Corollary~\ref{here} implies that, for any $\alpha\in(0,1)$,
\begin{equation*}
\begin{split}
&\| \partial_{x_n} u^{(\tau)}\|_{C^\alpha(B_{4/5}^{n-1}\times[-1/100,1/100])}\,\\ \le\,& C\,\Big(
\|f^{(\tau)}\|_{L^\infty(B_{9/10}^{n-1}/\times[-1/10,1/10])}
+\|u^{(\tau)}\|_{L^\infty(\R^n)}
\Big)
\\ \le\,& 
C\,\Big(
\|\partial_{x_n} f\|_{L^\infty(B_1)}+\|\partial_{x_n} u\|_{L^\infty(\R^n)}\Big)
,\end{split}\end{equation*}
for some~$C>0$.

Accordingly, sending~$\tau\to0$,
\begin{equation} \label{quasi11BIS}
\| \partial_{x_n}^2 u\|_{C^\alpha(B_{4/5}^{n-1}\times[-1/100,1/100])} \le
C\,\Big(
\|\partial_{x_n} f\|_{L^\infty(B_1)}+\|\partial_{x_n} u\|_{L^\infty(\R^n)}
\Big)
.\end{equation}
In addition, given~$x_n\in\R$, with~$|x_n|\le\frac{1}{100}$,
we define
\begin{eqnarray*}
\tilde f(x_1,\dots,x_{n-1}):=
f(x_1,\dots,x_{n-1},0)+\partial^2_{x_n} u(x_1,\dots,x_{n-1},0).\end{eqnarray*}
Notice that
\begin{equation}\label{parti1}
\begin{split}
& \|\tilde u\|_{C^{\alpha}(\R^{n-1})} \le
\| u\|_{C^{\alpha}(\R^{n-1}\times [-{1}/{100},{1}/{100}])}
\\ {\mbox{and }}\quad&
\|\tilde f\|_{C^{\alpha}(B_{4/5}^{n-1})}
\le \| f\|_{C^{\alpha}(B_{4/5}^{n-1}\times[-{1/100},{1}/{100}])}+
 \| \partial^2_{x_n} u\|_{C^{\alpha}(B_{4/5}^{n-1}\times [-{1/100},{1}/{100}])}
.\end{split}
\end{equation}
It holds that~$\tilde L\tilde u=\tilde f$ in~$B^{n-1}_{99/100}$,
for a suitable linear operator~$\tilde L$ which
satisfies formulas (1.1) and~(1.2) in~\cite{SRO}: for this fact,
see the proof of Corollary 1.3 in~\cite{FV} (and in particular
formula~(28) there). 
Consequently, we are in the position of using Theorem 1.1(b) in~\cite{SRO}.
In this way, we have that
$$ \|\tilde u\|_{C^{\alpha+2s}(B_{3/4}^{n-1})}\le C\,\Big(
\|\tilde u\|_{C^{\alpha}(\R^{n-1})}+
\|\tilde f\|_{C^{\alpha}(B_{4/5}^{n-1})}\Big),$$
for some~$C>0$, provided that~$\alpha+2s\not\in \N$.

{F}rom this and~\eqref{parti1} we obtain that
\begin{eqnarray*}
\|\tilde u\|_{C^{\alpha+2s}(B_{3/4}^{n-1})}&\le& C\,\Big(
\| f\|_{C^{\alpha}(B_{4/5}^{n-1}\times[-{1/100},{1}/{100}])}\\ &&\quad+
\| u\|_{C^{\alpha}(\R^{n-1}\times [-{1}/{100},{1}/{100}])}
+\| \partial^2_{x_n} u\|_{C^{\alpha}(B_{4/5}^{n-1}\times [-{1/100},{1}/{100}])}
\Big),\end{eqnarray*}
up to renaming~$C>0$. 

This and~\eqref{quasi11BIS} imply the desired result,
again, up to renaming~$C>0$.\end{proof}

Now we complete the proof of Corollary \ref{CC:x} by using a cut-off argument
as in the proof of Corollary \ref{Cut1}. Indeed, in that notation,
from \eqref{5wertfyguhjsgcxfscds} and
\eqref{4567ihgzf}, we can write \begin{equation}\label{765h5346fdsa57sdf}
L v = g\qquad{\mbox{in }}B_{1/2},\end{equation} 
with
$g$ as in \eqref{gme} and
$$ g_i(x):=
\int_{|x-z^{(i)}|\ge \frac1{10} }
\frac{(1-\tau)u(z^{(i)})}{|x-z^{(i)}|^{N_i+2s_i}}\,dz^{(i)}.$$
Notice that we can take derivatives in $x$ inside the integral,
hence, for any $j\in \N$ and $x\in B_{1/2}$,
$$ |D^j g_i(x)|\le C_j\,
\int_{|x-z^{(i)}|\ge \frac1{10} }
\frac{\|u\|_{L^\infty(\R^n)}}{|x-z^{(i)}|^{N_i+2s_i+j}}\,dz^{(i)}
\le C_j\,\|u\|_{L^\infty(\R^n)},$$
for suitable $C_j>0$.
In particular, we have that
$$ \|g\|_{C^\alpha(B_{4/5})}+
\|\partial_{x_n} g\|_{L^\infty(B_1)}
\le C\,\Big( \|f\|_{C^\alpha(B_{4/5})}+\|\partial_{x_n} f\|_{L^\infty(B_1)}
+ \|u\|_{L^\infty(\R^n)}\Big).$$
Hence, writing $\tilde v(x_1,\dots,x_{n-1}):=
v(x_1,\dots,x_{n-1},0)$,
from \eqref{765h5346fdsa57sdf} and Lemma \ref{39ygccvghfi11vbdhewi},
we obtain that
\begin{eqnarray*}
&& \|\tilde u\|_{C^{\alpha+2s}(B_{3/4}^{n-1})}=
\|\tilde u\|_{C^{\alpha+2s}(B_{3/4}^{n-1})}\\ &\le&
C\,\Big(
\| g\|_{C^{\alpha}(B_{4/5})}
+\|\partial_{x_n} g\|_{L^\infty(B_1)}
+\| v\|_{C^{\alpha}(\R^{n-1}\times [-{1}/{100},{1}/{100}])}
+\|\partial_{x_n} v\|_{L^\infty(\R^n)}
\Big)\\&\le&
C\,\Big(\| f\|_{C^{\alpha}(B_{4/5})}
+\|\partial_{x_n} f\|_{L^\infty(B_1)}+\|u\|_{L^\infty(\R^n)}
+\| u\|_{C^{\alpha}(Q)}
+\|\partial_{x_n} u\|_{L^\infty(Q)}
\Big),\end{eqnarray*}
where $Q$ is the support of the cut-off $\tau$.

Now, up to resizing balls, we may suppose that the original
equation was satisfied in some domain $Q'$, with $Q'\Supset Q$.
Hence, from formula~(8) of~\cite{FV}, we know that 
$$ \|\partial_{x_n} u\|_{L^\infty(Q)}\le C\,\Big(
\|f\|_{L^\infty(Q')}+\|u\|_{L^\infty(\R^n)}\Big),$$
and hence the result in Corollary \ref{CC:x} follows.

\vfill \end{document}